\documentclass[a4paper,11pt]{article}
\usepackage[nointlimits]{amsmath}
\usepackage{amssymb,amsmath,amsthm}
\usepackage{amsfonts}
\usepackage{latexsym}
\usepackage{graphicx}
\usepackage{color}
\usepackage{layout}
\usepackage[english]{babel}
\numberwithin{equation}{section} 

\def\ZZ{\mathbb Z}

\newcommand{\leg}[2]{\left({#1\over #2}\right)}
\newcommand{\fraz}[1]{\left\{#1\right\}}

\newtheorem{thm}{Theorem}[section]
\newtheorem{lem}[thm]{Lemma}
\newtheorem{cor}[thm]{Corollary}

\begin{document}

\title{\bf Congruences of multiple sums involving \\
invariant sequences under binomial transform}
\author{{\sc Roberto Tauraso}\\
Dipartimento di Matematica\\
Universit\`a di Roma ``Tor Vergata'', Italy\\
{\tt tauraso@mat.uniroma2.it}\\
{\tt http://www.mat.uniroma2.it/$\sim$tauraso}
}

\date{}
\maketitle
\begin{abstract}
\noindent We will prove several congruences modulo a power of a prime such as
$$
\sum_{0<k_1<\cdots<k_{n}<p}\leg{p-k_{n}}{3}
{(-1)^{k_{n}}\over k_1\cdots k_{n}}\equiv
 \left\{
\begin{array}{lll}
-{2^{n+1}+2\over 6^{n+1}}\,p\,B_{p-n-1}\left({1\over 3}\right)
&\pmod{p^2} &\mbox{if $n$ is odd}\\\\
-{2^{n+1}+4\over n6^n}\,B_{p-n}\left({1\over 3}\right)
&\pmod{p} &\mbox{if $n$ is even}\\
\end{array}
\right.
$$
where $n$ is a positive integer and $p$ is prime such that $p>\max(n+1,3)$.
\end{abstract}

\makeatletter{\renewcommand*{\@makefnmark}{}
\makeatother}

\section{Introduction} 
The classical binomial inversion formula states that the linear transformation 
of sequences
$$T(\{a_n\})=\left\{\sum_{k=0}^{n} {n\choose k}(-1)^{k} a_{k}\right\}$$
is an involution, that is $T\circ T$ is the identity map. 
Thus, $T$ has only two eigenvalues: $1$ and $-1$.
We denote by ${\cal S}_+$ and ${\cal S}_{-}$ the eigenspaces corresponding respectively 
to the eigenvalue $1$ and to the eigenvalue $-1$.
These eigenspaces contain many well known sequences: 
\begin{eqnarray*}
&&\{2^{-n}\},\;\{L_{n}\},\;
\{(-1)^n B_{n}\},\;\{(n+1)C_{n}4^{-n}\}\in  {\cal S}_{+},\\
&&\{0,1,1,\cdots\},\;\{F_n\},\;
\{(-1)^n\leg{n}{3}\},\;\{2^n-(-1)^n\}\in \cal{S}_{-}
\end{eqnarray*}
where $\{F_n\}$, $\{L_n\}$, $\{B_n\}$, $\{C_n\}$ denote respectively
the Fibonacci, Lucas, Bernoulli and Catalan numbers.
For a more detailed analysis of the properties of $ {\cal S}_{+}$ and $ {\cal S}_{-}$
the reader is referred to \cite{Sunzh:01} and \cite{Wa:05}.

In this note we would like to present several congruences
of multiple sums which involves these invariant sequences. Our main result is the following.

\begin{thm} Let $\{a_n\}\in \cal{S}_{-}$.
Let $n$ be a positive odd integer and let $p$ be a prime such that $p>n+1$ then
$$\sum_{0<k_1<\cdots<k_{n}<p}
{a_{p-k_{n}}\over k_1\cdots k_{n}}\equiv
{p(n+1)\over 2}\, \sum_{0<k_1<\cdots<k_{n+1}<p}
{a_{p-k_{n+1}}\over k_1\cdots k_{n+1}}\pmod{p^3}.$$
\end{thm}

Note that since
$$\sum_{k=0}^{n} {n\choose k}(-1)^{k} ka_{k-1}=-n
\sum_{k=0}^{n-1} {n-1\choose k}(-1)^{k} a_{k}$$
then $\{a_{n}\}\in \cal{S}_{+}$ if and only if $\{na_{n-1}\}\in \cal{S}_{-}$.
Hence by the previous theorem we easily find that

\begin{cor} 
Let $n$ be a positive odd integer and let $p$ be a prime such that $p>n+1$. 

i) If $\{a_n\}\in \cal{S}_{-}$.
$$\sum_{0<k_1<\cdots<k_{n}<p}
{a_{p-k_n}\over k_1\cdots k_{n}}\equiv 
\sum_{0<k_1<\cdots<k_{n}<p}
{a_{k_1}\over k_1\cdots k_{n}}\equiv 0 \pmod{p}.$$

ii) If $\{a_n\}\in \cal{S}_{+}$.
$$\sum_{0<k_1<\cdots<k_{n}<p}
{a_{p-k_n-1}\over k_1\cdots k_{n-1}}\equiv 
\sum_{0<k_1<\cdots<k_{n}<p}
{a_{k_1-1}\over k_2\cdots k_{n}}\equiv 0 \pmod{p}.$$
\end{cor}

In the final section we consider the special class of
sequences in ${\cal S}_{-}$ which are second-order linear recurrences
and we will prove the congruence mentioned in the abstract which generalizes a 
result established in \cite{ZhSuzw:09}.

\section{The main result}
The starting point of our work is the following identity.

\begin{lem} 
Let $\{a_n\}\in {\cal S}_{-}$ then for $m,n\geq 0$  
$$\sum_{k=0}^{n}\left[{(m-1)n+k-1 \choose k} 
+(-1)^{n-k}{m n \choose k}\right]a_{n-k}=0.$$
\end{lem}

\begin{proof} This identity can be obtained from
$$\sum_{k=0}^{n}{n\choose k}
\left[f_k+(-1)^{n-k} \sum_{i=0}^k {k\choose i}f_i\right]a_{n-k}=0$$
which appears in \cite{Sunzh:01} (see also \cite{Wa:05}), 
by taking $f(k)={(m-1)n+k-1 \choose k}/{n\choose k}$.

\noindent However, for completeness sake, we give here a direct proof.
Let $A(z)$ be the generating function of $\{a_n\}$, then
$$A(z)=-{1\over 1-z} A\left({z \over z-1}\right).$$
because $\{a_{n}\}\in \cal{S}_{-}$.
Since 
$${(m-1)n+k-1 \choose k}=[z^{k}]{1\over (1-z)^{(m-1)n}}=
[z^{-1}]{1\over z^{n+1}(1-z)^{(m-1)n}}\cdot z^{n-k}
$$
and
$$
(-1)^{n-k}{m n \choose k}=[z^{k}]{(-1)^{n-k}\over (1-z)^{mn-k+1}}=
[z^{-1}]{1\over z^{n+1}(1-z)^{(m-1)n+1}}\cdot \left({z \over z-1}\right)^{n-k},
$$
then the right-hand side of the identity becomes
$$[z^{-1}]{1\over z^{n+1}(1-z)^{(m-1)n}}\cdot
\left[A(z)+{1\over 1-z} A\left({z \over z-1}\right)\right]=0.$$
\end{proof}

Before proving our main result we 
define the {\it multiple harmonic sum} of order $n>0$ as
$$H_r^{(n)}=\sum_{1\leq k_1<k_2<\dots<k_n\leq r}
{1\over k_1 k_2 \cdots k_n} \quad\mbox{for $r\geq1$}.$$

\begin{proof}[Proof of Theorem 1.1] 
Since $a_0=0$ and $p$ is odd, by the previous lemma we have that
$$\sum_{k=1}^{p-1}\left[{(m-1)p+k-1 \choose k} 
-(-1)^{k}{mp \choose k}\right]a_{p-k}=0.$$
By expanding the binomial coefficients for $1<k<p$ we get
$${(m-1)p+k-1\choose k}=
{(m-1)p\over k}\prod_{j=1}^{k-1}\left(1+{(m-1)p\over j}\right)=
 {1\over k}\sum_{j\geq 1}((m-1)p)^jH_{k-1}^{(j-1)}$$
and
$${mp\choose k}={mp\over k}(-1)^{k-1}\prod_{j=1}^{k-1}\left(1-{mp\over j}\right)
={(-1)^k\over k}\sum_{j\geq 1}(-mp)^jH_{k-1}^{(j-1)}.$$
Hence, 
$$\sum_{j\geq 1}((m-1)^j-(-m)^j)\,p^jS_j=0$$
where $S_j=\sum_{k=1}^{p-1}H_{k-1}^{(j-1)} {a_{p-k}\over k}$.
The infinite matrix of the coefficients of $\{p^jS_j\}$ is
$$
\begin{bmatrix}
1&-1&1&-1&1&-1&1&-1&\cdots\\
3&-3&9&-15&33&-63&129&-255&\cdots\\
5&-5&35&-65&275&-665&2315&-6305&\cdots\\
7&-7&91&-175&1267&-3367&18571&-58975&\cdots\\
9&-9&189&-369&4149&-11529&94509&-325089&\cdots\\
\vdots&\vdots&\vdots&\vdots&\vdots&\vdots&\vdots&\vdots&\ddots\\
\end{bmatrix}
$$
and by performing the Gaussian elimination we obtain
$$
\begin{bmatrix}
1&-1&1&-1&1&-1&1&-1&\cdots\\
0& 0&1&-2&5&-10&21&-42&\cdots\\
0&0&0&0&1&-3&14&-42&\cdots\\
0&0&0&0&0&0&1&-4&\cdots\\
0&0&0&0&0&0&0&0&\cdots\\
\vdots&\vdots&\vdots&\vdots&\vdots&\vdots&\vdots&\vdots&\ddots\\
\end{bmatrix}.
$$
In general let 
$$t(i,j)={2\over(2i)!}\sum_{m=1}^{i}(-1)^{i-m}{2i\choose i+m} m^j$$
then
\begin{eqnarray*}
0&=&{1\over(2i)!}\sum_{m=1}^{i}(-1)^{i-m}(i+m){2i\choose i+m}\sum_{j\geq 1}((m-1)^j-(-m)^j)\,p^jS_j\\
&=&\sum_{j\geq 1} {p^jS_j\over(2i)!}\left[
-\sum_{m=1}^{i}(-1)^{i-(m-1)}(i-(m-1)){2i\choose i+m-1}(m-1)^j\right.\\
&&\;\;\;\;\;\;\;\;\;\;\;\;\;\;\;\;\;\;\;\;
\left.-(-1)^j\sum_{m=1}^{i}(-1)^{i-m}(i+m){2i\choose i+m}m^j
\right]\\
&=&\sum_{j\geq 1} {p^jS_j\over 2}\left[-it(i,j)+t(i,j+1)-i(-1)^j t(i,j)-(-1)^j t(i,j+1)\right]\\
&=&\sum_{j\geq 1} p^jS_j\left[-{1+(-1)^j\over 2}it(i,j)+{1-(-1)^j\over 2}t(i,j+1)\right]\\
&=&\sum_{j\geq 1} p^{2j-1}t(i,2j)\left(S_{2j-1}-ipS_{2j}\right).
\end{eqnarray*}
Note that $t(i,2j)$ are the {\sl triangle central factorial numbers} 
(see \cite{Ri:68} p. 217)
therefore $t(i,2j)=0$ if $i<2j$ and $t(j,2j)=1$. Thus 
$$S_{2i-1}\equiv ipS_{2j}+p^2t(i,i+1)S_{2i+1} \pmod{p^3}.$$
It follows that $S_{n}\equiv 0 \pmod{p}$ when $n$ is odd and finally we get
$$S_{2i-1}\equiv ip\,S_{2i}\pmod{p^3}.$$
\end{proof}

\section{A special class of invariant sequences}
Let us prove first a preliminary lemma.

\begin{lem} Let $n$ be a positive integer and let $p$ be any prime such that $p>n+1$.
Then the two polynomials of $\ZZ_p[x]$
$$G_n(x)=\sum_{k=1}^{p-1} H_{k-1}^{(n-1)}{x^k\over k}, 
\quad\mbox{and}\quad g_n(x)=\sum_{k=1}^{p-1} {x^k\over k^n}.$$
satisfy the congruence
$$G_n(x)\equiv (-1)^{n-1} g_n(1-x) \pmod{p}.$$
\end{lem}

\begin{proof} We prove the congruence by induction on $n$. 

\noindent For $n=1$, since 
${p\choose k}=(-1)^{k-1}{p\over k}\pmod{p^2}$ for $0<k<p$ then
$$G_1(x)\equiv {1\over p}\sum_{k=1}^{p-1}(-1)^{k-1}{p\choose k}{x^k}
=-{1\over p}\sum_{k=1}^{p-1}{p\choose k}{(-x)^k}=
{(1-x)^p-1-x^p\over p}\pmod{p}.$$
Hence $G_1(x)\equiv G_1(1-x)=g_1(1-x)\pmod{p}$.

\noindent Assume that $n>1$. 
The formal derivative yields
\begin{eqnarray*}
{d\over dx} G_n(x)&=&\sum_{k=1}^{p-1} H_{k-1}^{(n-1)}x^{k-1}
=\sum_{k=1}^{p-1}\sum_{j=1}^{k-1} H_{j-1}^{(n-2)}{x^{k-1}\over j}\\
&=&\sum_{j=1}^{p-1} H_{j-1}^{(n-2)}{1\over j}\sum_{k=j+1}^{p-1}x^{k-1}
=\sum_{j=1}^{p-1} H_{j-1}^{(n-2)}{1\over j}
\cdot {x^{p-1}-x^{j}\over x-1}\\
&=&{x^{p-1}\over x-1}\, H_{p-1}^{(n-1)}-{G_{n-1}(x)\over x-1} 
\equiv {G_{n-1}(x)\over 1-x} \pmod{p}
\end{eqnarray*}
where in the last step we used the fact that 
$H_{p-1}^{(n-1)}\equiv 0 \pmod{p}$ (see for example \cite{Zh:08}). Moreover
$${d\over dx} g_n(1-x)=-\sum_{k=1}^{p-1} {(1-x)^{k-1}\over k^{n-1}}
=-{g_{n-1}(1-x)\over 1-x}.$$
Hence, by the induction hypothesis
$$(1-x){d\over dx} \left(G_n(x)+(-1)^n g_n(1-x)\right)\equiv
G_{n-1}(x)+(-1)^{n-1} g_{n-1}(1-x)\equiv 0 \pmod{p}.$$
Thus $G_n(x)+(-1)^n g_n(1-x)\equiv c_1$ (mod $p$) for some constant $c_1$ since this polynomial
has degree $<p$. By letting $x=1$ we find that 
$$G_n(x)+(-1)^n g_n(1-x)\equiv c_1\equiv G_n(1)+(-1)^n g_n(0)
=H_{p-1}^{(n)}\equiv 0 \pmod{p}.$$
\end{proof}

Let us consider a sequence $\{a_n\}\in {\cal S}_{-}$ which is a second-order recurrence. 
Since its generating function should satisfy the identity
$$A(z)=-{1\over 1-z}\, A\left({z \over z-1}\right)$$
it is easy to verify that 
$$A(z)={a_1 z\over 1-z-cz^2}$$
for some real number $c$.  The following result holds for any invariant sequence of this kind.

\begin{thm} Let $n$ be a positive integer and let $p$ be a prime such that $p>n+1$.
If $\sum_{k\geq 0}a_k\,z^k={a_1 z\over 1-z-cz^2}$ for some integer $c$ then
$$
\sum_{0<k_1<\cdots<k_{n}<p}
{a_{p-k_{n}}\over k_1\cdots k_{n}}\equiv
 \left\{
\begin{array}{lll}\displaystyle
-p(n+1)\sum_{k=1}^{(p-1)/2}{c^k a_{p-2k}\over k^{n+1}} &\pmod{p^2} 
&\mbox{if $n$ is odd}\\\\\displaystyle
-2\sum_{k=1}^{(p-1)/2}{c^k a_{p-2k}\over k^n} &\pmod{p} 
&\mbox{if $n$ is even}
\end{array}
\right. .
$$
\end{thm}
\begin{proof} By linearity we can assume that $a_1=1$ (the case $a_1=0$ is trivial).

\noindent Since $\{a_k\}\in {\cal S}_{-}$ then, by Theorem 1.1, if $n$ is odd we have that
$$\sum_{0<k_1<\cdots<k_{n}<p}
{a_{p-k_{n}}\over k_1\cdots k_{n}}\equiv
{p(n+1)\over 2} \sum_{0<k_1<\cdots<k_{n+1}<p}
{a_{p-k_{n+1}}\over k_1\cdots k_{n+1}}\pmod{p^3},$$
thus it suffices to consider the case when $n$ is even.

\noindent Let $\Delta=1+4c$ and consider the ring $\ZZ_p[\sqrt{\Delta}]$.
Then for $k\geq 0$
$$a_k={w_{+}^k+w_{-}^k\over \sqrt{\Delta}}\quad\mbox{with}\quad
w_{\pm}={1\pm\sqrt{\Delta}\over 2}.$$
This formula allows to extend the sequence $\{a_k\}$ also for 
negative indeces: $a_{-k}=-a_{k}(-c)^{-k}$ for $k>0$ because $w_{+}+w_{-}=1$ and $w_{+}w_{-}=-c$.
Therefore, by the previous lemma, we have that
\begin{eqnarray*}
\sum_{0<k_1<\cdots<k_{n}<p}
{a_{p-k_{n}}\over k_1\cdots k_{n}}
&=&\sum_{k=1}^{p-1} H_{k-1}^{(n-1)}{a_{p-k}\over k}\\
&=&{w_{+}^p\over \sqrt{\Delta}}\sum_{k=1}^{p-1} H_{k-1}^{(n-1)}{w_{+}^{-k}\over k}
-{w_{-}^p\over \sqrt{\Delta}}\sum_{k=1}^{p-1} H_{k-1}^{(n-1)}{w_{-}^{-k}\over k}\\
&\equiv&
-{w_{+}^p\over \sqrt{\Delta}}\sum_{k=1}^{p-1} {(1-w_{+}^{-1})^k\over k^n}
+{w_{-}^p\over \sqrt{\Delta}}\sum_{k=1}^{p-1}{(1-w_{-}^{-1})^k\over k^n}
\pmod{p}.
\end{eqnarray*}
Since $1-w_{\pm}^{-1}=cw_{\pm}^{-2}$, it follows that
\begin{eqnarray*}
\sum_{0<k_1<\cdots<k_{n}<p}
{a_{p-k_{n}}\over k_1\cdots k_{n}}
&\equiv&
-{w_{+}^p\over \sqrt{\Delta}}\sum_{k=1}^{p-1}
{c^k w_{+}^{-2k}\over k^n}
+{w_{-}^p\over \sqrt{\Delta}}\sum_{k=1}^{p-1}{c^k w_{-}^{-2k}\over k^n}\\
&\equiv&
-\sum_{k=1}^{p-1}{c^k a_{p-2k}\over k^n}\\
&\equiv&
-\sum_{k=1}^{(p-1)/2}{c^k a_{p-2k}\over k^n}
-\sum_{k=1}^{(p-1)/2}{c^{p-k} a_{-(p-2k)}\over (p-k)^n}\\
&\equiv&
-2\sum_{k=1}^{(p-1)/2}{c^k a_{p-2k}\over k^n}\pmod{p}.
\end{eqnarray*}
\end{proof}

Finally, we prove are ready to prove the congruence mentioned in the abstract.

\begin{thm} Let $n$ be a positive integer and let $p$ be a prime such that $p>\max(n+1,3)$. Then 
$$
\sum_{0<k_1<\cdots<k_{n}<p}\leg{p-k_{n}}{3}
{(-1)^{k_{n}}\over k_1\cdots k_{n}}\equiv
 \left\{
\begin{array}{lll}
-{2^{n+1}+2\over 6^{n+1}}\,p\,B_{p-n-1}\left({1\over 3}\right)
&\pmod{p^2} &\mbox{if $n$ is odd}\\\\
-{2^{n+1}+4\over n6^n}\,B_{p-n}\left({1\over 3}\right)
&\pmod{p} &\mbox{if $n$ is even}\\
\end{array}
\right. .
$$
\end{thm}

\begin{proof} Let $a_k=(-1)^{p-k}\leg{k}{3}$ then its generating function is
$$A(z)={z\over 1-z+z^2}$$
and by the previous theorem, for $n$ even the right-hand side yields
$$
\sum_{0<k_1<\cdots<k_{n}<p}\leg{p-k_{n}}{3}
{(-1)^{k_{n}}\over k_1\cdots k_{n}}
\equiv
-2\sum_{k=1}^{(p-1)/2}{(-1)^k\over k^n}\leg{p-2k}{3}
\equiv\sum_{k=1}^{p-1}{(-1)^{k-1}\over k^n}\leg{p+k}{3}\pmod{p}.
$$
By \cite{Sunzh:08} we have that 
$$\sum_{\scriptsize
\begin{array}{c} 
k=1\\ k\equiv r\; \mbox{\scriptsize(mod $6$)}
\end{array}}^{p-1}
\!\!\!\!{1\over k^n}\equiv {1\over n6^n}\left(
B_{p-n}\left(\fraz{{r\over 6}}\right)
-B_{p-n}\left(\fraz{{r-p\over 6}}\right)\right)\pmod{p}
$$
for $p>\max(n+1,3)$ and for $r=0,1,2,3,4,5$.
Moreover, the Bernoulli polynomials satisfy the reflection property and 
the multiplication formula (see for example \cite{IrRo:90} p.248):
$$B_m(1-x)=(-1)^m B_m(x)\quad\mbox{and}\quad
B_m(ax)=a^{m-1} \sum_{k=0}^{a-1}B_m\left(x+{k\over a}\right)\quad\mbox{for $m,a>0$.}$$
Hence, since $p-n$ is odd, 
\begin{eqnarray*}
&&B_{p-n}\left(0\right)=B_{p-n}\left({1\over 2}\right)=0\,,\quad B_{p-n}\left({2\over 3}\right)=-B_{p-n}\left({1\over 3}\right)\\
&&B_{p-n}\left({1\over 6}\right)=-B_{p-n}\left({5\over 6}\right)
=\left(1+2^{n-(p-1)}\right)B_{p-n}\left({1\over 3}\right).
\end{eqnarray*}
By these preliminary remarks it is easy to verify that
$$\sum_{k=1}^{p-1}{(-1)^{k-1}\over k^n}\leg{p+k}{3}\equiv
-{2^{n+1}+4\over n6^n}\,B_{p-n}\left({1\over 3}\right) \pmod{p}.$$
On the other hand, if $n$ is odd, by Theorem 1.1, it follows from the above congruence that
\begin{eqnarray*}
\sum_{0<k_1<\cdots<k_{n}<p}\leg{p-k_{n}}{3}
{(-1)^{k_{n}}\over k_1\cdots k_{n}}
&\equiv& {p(n+1)\over 2}
\sum_{0<k_1<\cdots<k_{n+1}<p}\leg{p-k_{n+1}}{3}
{(-1)^{k_{n}}\over k_1\cdots k_{n+1}}\pmod{p^3}\\
&\equiv& -{2^{n+1}+2\over 6^{n+1}}\,pB_{p-n-1}\left({1\over 3}\right) \pmod{p^2}.
\end{eqnarray*}
\end{proof}


\end{document}